\documentclass[a4paper,12pt,one column]{article}
\usepackage{amsmath}
\usepackage{amssymb}
\usepackage{color}
\usepackage{amsthm, amscd, mathrsfs}
\usepackage{mathtools}
\newenvironment{FG}{\noindent\color{blue} FG : }{}

\newtheorem{thm}{Theorem}[section]

\newtheorem{lem}[thm]{Lemma}

\newtheorem{cor}[thm]{Corollary}

\theoremstyle{remark}
\newtheorem{rem}[thm]{\bf Remark}

\newtheorem{exmp}[thm]{\bf Example}

\newtheorem{defn}[thm]{\bf Definition}
\newcommand{\la}{\mathcal}
\newcommand{\mbb}{\mathbb}
\newcommand{\fra}{\mathfrak}

\newcommand{\ovr}{\overrightarrow}

\newcommand{\CN}{\la{N}}
\newcommand{\N}{\mathbb{N}}
\newcommand{\Z}{\mbb{Z}}
\newcommand{\Q}{\mbb{Q}}
\newcommand{\R}{\mbb{R}}

\newcommand{\Set}[2]{\left\{\, #1 \;\middle|\; #2 \,\right\}}

\newcommand{\rb}[1]{{\left( #1 \right)}}

\newcommand{\ds}{\displaystyle}
\usepackage{fancyhdr}
\makeatletter
\let\@fnsymbol\@arabic
\makeatother

\begin{document}
\author{Funda Gul, Alexei G. Myasnikov\thanks{The second author was partially supported by NSF grant
DMS-1318716 and by Russian Research Fund, project 14-11-00085.}, Mahmood Sohrabi}
\title{Distortion of embeddings of a torsion-free finitely generated nilpotent group into a unitriangular group}
\date{}
\maketitle
\begin{abstract} In this paper we study distortion of various well-known embeddings of finitely generated torsion-free nilpotent groups $G$ into  unitriangular groups $UT_n(\mathbb{Z})$. We also provide a polynomial time algorithm  for finding distortion of a given subgroup of $G$.

\noindent
\textbf{Keywords.} Nilpotent groups, subgroup distortion, Nickel's embedding, Jennings' embedding, unitriangular groups, Mal'cev basis, Heisenberg group. 
\end{abstract}

\section{Introduction} Studying groups from a geometric point of view has been a major research theme in the past few decades. Not only has the geometric point of view created new questions and insights in understanding groups, but also it has given us some further criteria for evaluating classical constructions. In this work we study embeddings of  finitely generated torsion free nilpotent groups into groups of unitriangular matrices over integers from geometric point of view.  There are various emebddings of this type, some of them are known for more than half a century already, and for pure algebraic  purposes they seem equally good. 
However, with developments in algorithmic and geometric group theory new principal questions arise. The first group of questions concerns with the algorithmic properties of embeddings, for example, one may need to know how much time is required  when given a finitely generated torsion free nilpotent  group $G$  to find its image in $UT_n(\mathbb{Z})$ under a given embedding. The second group of questions is about geometric properties of embeddings, in particular, one may want to know the degree of distortion of the image of $G$ in $UT_n(\mathbb{Z})$ under a given embedding. In this line of thought an embedding of groups $\phi: G \to K$ is ``good'' if $\phi(G)$ is an undistorted subgroup of $K$. Below  we address the latter questions for some standard embeddings of $G$ into integer unitriangular matrices. 

One of the earliest, most studied, and explicit embeddings of a finitely generated torsion-free nilpotent group $G$ into a unitriangular group is due to S. A. Jennings.   In~\cite{J}, S. A. Jennings offers a detailed study of the group ring of $G$  and the embedding follows as a natural corollary. P. Hall's Edmonton's notes~\cite{hall} contain  an exposition of the later. We show first that the standard Jennings embedding distorts the image of $G$, then we further explain and prove that the embedding can not be made without distortion for $\tau$-groups with rank of $\Gamma_c(G) \geq 2$, where $c$ is the nilpotency class of $G$. We also show that the image of  $(2n+1)$-dimensional Heisenberg group for $n\geq2$ under Jenning's embedding is always distorted. Finally, we show that  in case of unitraingular groups if a proper Mal'cev basis (non-standard) is selected, the embedding can be made without distortion.

In the book  ~\cite{km} Kargapolov and Merzlyakov provide another embedding of $G$ into unitriangular groups, based on the method of coordinates.  However, the construction itself is rather theoretical, it is obscured by the inductive arguments, and is hard to build the image of $G$ in an "algorithmic" manner. Recently in \cite{N} W. Nickel gave another embedding for $G$ into unitriangular groups. His embedding is algorithmic friendly, in fact the algorithm is implemented and can be found in \cite{N}. In this paper, similar to Jennings' we show that when the non-standard Mal'cev basis for the unitriangular group is selected, Nickel's embedding can be made undistorted. We also show that for $n\geq2$ the image of any $(2n+1)$-dimensional Heisenberg group under Nickel's embedding is always distorted in $UT_d(\Z)$ for $d=2n+2$.

The notion of distortion of a subgroup of a given group has been widely studied in recent years. In the context of nilpotent groups the main results are due to D. V. Osin~\cite{osin}. He gives a thorough study of distortion of subgroups of a finitely generated nilpotent group, through the study of the notion in connected simply-connected Lie groups. We shall use one of his main results here. Notice that T. Davis~\cite{D} also provides a criterion for a subgroup of a free nilpotent group being undistorted.

Finally, we provide a polynomial time algorithm for finding distortion of a given subgroup of a finitely generated nilpotent group $G$.

\subsection{Distortion of subgroups in unitriangular groups}\label{distortion}

Here we provide the required definitions as well as D. Osin's result on distortion of subgroups of finitely generated nilpotent groups. If $R$ is an associative ring with unity 1, then by $UT_n(R)$ we mean the group of upper triangular $n\times n$ matrices over $R$ with all diagonal entries equal to 1. We also call such a group a \emph{unitriangular group}.  
\begin{defn}\label{OS:dist:def} Let $H\leq G$ and $dist_G: G\times G\to \R_{+} \cup \{0\}$ be the word metric on $G$ and the word metric on $H$ be denoted by $dist_H$. The ball of radius $n$ in $G$ is defined as:
$$B(n) = \{ g\in G | \: dist_G(1,g) \leq n\} $$
and the distortion function (of $H$ as a subgroup of $G$) is defined by:
$$ \Delta_H ^G (n)= max\{dist_H(1,h):h\in H \cap B(n) \} $$
\end{defn}
Given a nilpotent group $G$ we set $\Gamma_1(G)=_{\text{def}}G$ and $\Gamma_{i+1}(G)=_{\text{def}}[\Gamma_i(G),G]$ for $i\geq 2$ where for subgroups $A$ and $B$ of $G$, $[A,B]$ denotes the subgroup generated by the commutators $[x,y]=x^{-1}y^{-1}xy$, $x\in A$ and $y\in B$. The subgroup $\Gamma_i=\Gamma_i{(G)}$ is called the \emph{i'th term of the lower central series of $G$}. The \emph{nilpotency class} of $G$ is the least number $c$ for which $\Gamma_{c+1}=\{1\}$.
 
\begin{defn}\label{Os:weight-tor:defn} Let $G$ be a finitely generated nilpotent group and $G^0$ be the set of all elements of infinite order. For any $g\in G^0$, the weight $\nu_{G}(g)$ is the maximal $k$ such that $\langle g \rangle \cap \Gamma_k(G) \neq \{1\}.$
\end{defn}

As a corollary of Theorem~2.2 in \cite{osin}, one can obtain the following theorem.  

\begin{thm}[D. Osin \cite{osin}]\label{utnZR} Let $G$ be finitely generated torsion free nilpotent group then the distortion of $H\leq G$ is equivalent to
 $$ \Delta_H^G (n) \sim n^{r}$$ where $$ r=\max_{{h\in H\setminus\{1\}}} \frac{\nu_G(h)}{\nu_H(h)}$$
\end{thm} 

\subsection{Main results and the structure of the paper}
 
In Section~\ref{UTN}, we show that every rational number of the form $\frac{p}{q}$ for $p\geq q \geq 2$ can be realized as the distortion of some subgroup of $UT_N(\Z)$, for some $N \geq 3$. In Section~\ref{embedding}, we fix our notation and discuss both ''Jennings embedding'' and ''Nickel's Embedding'' of a finitely generated torsion-free nilpotent groups into $UT_d(\Z)$ for some $d$. Section~\ref{main} includes the theorems regarding the distortions of the embeddings. Finally, in Section~\ref{subgroup}, we provide a polynomial time algorithm for finding distortion of a given subgroup of finitely generated nilpotent group $G$.

\section{Rationally distorted subgroups of $UT_N(\Z)$}\label{UTN}
In our first section we like to show that every rational number of the form $\frac{p}{q}$ for $p\geq q \geq 2$ can be realized as the distortion of some subgroup of $UT_N(\Z)$ for some $N \geq 3$.

We will denote each generator of $UT_N(\Z)$ in terms of $s_{ij}$ which represents $N\times N$ upper unitriangular matrix with $ij$'th entry $1$. More specifically, if $e_{ij}(\alpha)$, $i<j$, is the matrix with $ij'th$ entry $\alpha$ and the rest of the entries $0$, one can see that $s_{ij}(\alpha)=1+e_{ij}(\alpha)$ and $s_{ij}=s_{ij}(1)$ as well as
$$s_{ij}^{-1}=\left(s_{ij}(1)\right)^{-1}= s_{ij}(-1),$$ $$[s_{ij}, s_{jk}]=s_{ik}, \: \text{and}\: [s_{ji}, s_{kj}]=s_{ik}(-1) \: \text{for} \: i<j<k.$$ 

\begin{thm}
For any rational number $\frac{p}{q}$, where $p\geq q \geq 2$, there is a subgroup $H$ of a unitriangular group $G=UT_N(\Z)$ for some $N \geq 3$ such that 
$$ \Delta_H^G (n) \sim n^{p/q}.$$
\end{thm}

\begin{proof}

For a given rational number $\frac{p}{q}$ for $p \geq q$, let $G= UT_N(\Z)$ such that $N=p+1$ and 
$$ H= \langle s_{12}, s_{23}, \dots,s_{(m-2)(m-1)}, s_{(N-1)N}s_{(m-1)N}, s_{1N}  \rangle \leq UT_N(\Z)$$
where $m=q+1$ and $3 \leq m < N$. 

Claim: 
$$ \Delta_H^G (n) \sim n^{p/q}. $$
To prove our claim, we will use Theorem~\ref{utnZR} and show that 
 $$ r= \max_{h\in H} \frac{\nu_G(h)}{\nu_H(h)}= \frac{p}{q}.$$
 
Note that, for all $2\leq k \leq m-1$,
$$\nu_G(s_{(k-1)k})= \nu_H(s_{(k-1)k})=1$$
and
$$ \nu_G(s_{(N-1)N}s_{(m-1)N})= \nu_H(s_{(N-1)N}s_{(m-1)N})=1.$$
Now we have
$$ [s_{12},s_{23}, \dots, s_{(m-2)(m-1)}, s_{(N-1)N}s_{(m-1)N}]= s_{1N}$$ 
in $H$, which shows that
$$\nu_H({s_{1N}})=m-1.$$ 
As a result we get
$$ r= \max_{h\in H-{1}} \frac{\nu_G(h)}{\nu_H(h)}= \frac{N-1}{m-1}=\frac{p}{q}$$
since $\nu_G(s_{1N})= N-1$. 
\end{proof}

\section{Embeddings of $\tau$-groups into $UT_n(\Z)$}\label{embedding}
\subsection{Jennings' Embedding}
By a $\tau$-group we mean a finitely generated torsion-free non-abelian nilpotent group. Below we will describe embeddings of $\tau$-groups into unitriangular groups over the ring of integers $\Z$. We first use the embedding provided by Jennings in his analysis of dimension subgroups of an arbitrary group $G$. Our presentation here is taken from P. Hall's Edmonton Notes~\cite{hall}. We refer to this reference for further details.  One of the main goals of this section is to set up the notation that we use in the sequel.

Let $G$ be a group and $\fra{R}$ any field of characteristic $0$.
Consider the group algebra
$$
\fra{R}G= \Set{ \sum_{x\in G} \lambda_{x} x}{\lambda_{x} \in \fra{R} \: and \: \lambda_{x}=0\: \text{for all but finitely many x} }.
$$
The set
$$I=\Set{\sum_{x\in G}\lambda_{x} x \in \fra{R}G}{\sum_{x\in G} \lambda_{x}=0 }$$
is an ideal in $\fra{R}G$ called the \emph{augmentation ideal}.
We usually identify $G$ with its image in $\fra{R}G$ under the embedding $g \to 1\cdot g$.

The subgroup
$$\Delta_{n}(G)= G\cap (1+I^n)$$
where $I^n=\{ \sum \alpha_{i}u_{1} \dots u_{n} : u_{i} \in I ~\text{and}~ \alpha_{i} \in \fra{K}\}$,
is called the \emph{n'th dimension subgroup of $G$}.
Notice that $G$ acts on $\fra{R}G$ by right multiplication, which induces an action of $G$ on $\fra{R}G/I^n$.
So, one has  a homomorphism
$$
\phi_n : G \to Aut\left( \fra{R}G/I^n \right),$$
where $Aut\rb{\fra{R}G/I^n}$ is  the group of nonsingular $\fra{R}$-linear transformations of the $\fra{R}$-vector space $\fra{R}G/I^n$ (which is finite dimensional over $\fra{R}$ as we shall see below). It is not hard to see that 
$$\Delta_{n}(G)=\ker(\phi_n).$$

Given a group $G$ and a subgroup $N$ of $G$ the \emph{isolator of $N$ in $G$} is the set:
$$
Is(N) = \Set{g\in G}{g^m\in N \mbox{ for some } 0\neq m \in\N}.
$$
When $G$ is nilpotent $Is(N)$ is a subgroup and normal in $G$ if $N$ is a normal subgroup of $G$.

Given a nilpotent group $G$ of class $c$, $\tau_i(G)$, denotes the \emph{$i$'th term of the isolated lower central series of $G$}, i.e.
$$\tau_i(G)=_{\text{def}}Is(\Gamma_i(G)).$$

\begin{defn} A basis $U = \{u_i\}$ for the $\fra{R}$-vector space $\fra{R}G$ is called an \emph{integral basis} if

\begin{enumerate}
\item[(i)] Each $u_i$ can be written as $u_i= \sum_{x\in G} n_{x}x$, such that $n_x \in \Z$ for all $x$.
\item[(ii)] Every $x\in G$ can be written as $x= \sum_{i} m_iu_i$, such that  $m_i \in \Z$ for all $i$.
\end{enumerate}
\end{defn}
Note: If $U$ is  an integral basis of $\fra{R}G$ then for every $u_i,u_j \in U$ one has $u_iu_j=\sum c_{ijk}u_k$ for some $c_{ijk} \in \Z$ (called the multiplicative constants of $U$).

From now on, we will assume that $G$ is a $\tau$-group and $n=c+1$, where $c$ is the nilpotency class of $G$.
Consider the isolated lower central series
$$ G= \tau_1(G) > \tau_2(G) > \dots > \tau_{n-1}(G) > \tau_n(G)=1 $$ of $G$. Refine it to a poly-infinite-cyclic  series of $G$:
$$ G= G_0>G_1>G_2> \dots > G_M=1$$ 
where $G_{i-1}/G_i \simeq \Z$ and chose $x_i \in G_{i-1}$  such that $G_{i-1}= \langle x_i, G_i \rangle$. Then the tuple $x_1, \ldots, x_M$ forms a  Mal'cev basis of $G$ associated with the poly-infinite-cyclic series above. 

\begin {lem}\label{basis:kG:lem} Let $G$ be a $\tau$-group and $\fra{R}$ a field of characteristic zero.  Consider a Mal'cev basis $(x_1, x_2, \ldots, x_M)$ for $G$ obtained as above and set $u_i= 1-x_i$. Then the set $V$ of all products $v= v_1v_2 \dots v_M$ in which each $v_i$ has one of the following forms,
\begin {itemize}
\item $u_{i}^{r_i}, \quad r_i\geq 0$
\item $u_{i}^nx_{i}^{-s_i}, \quad s_i >0$.
\end{itemize}
forms an integral basis for $\fra{R}G$.
\end{lem}

\begin{defn} Let $V$ be a basis of $\fra{R}G$  described in Lemma~\ref{basis:kG:lem}. For an element  $v=v_1v_2\dots v_M \in V$ define the \emph{weight, $\mu(v)$,}  as follows:
\begin{itemize}
\item $\mu (v)\geq n$ if at least for one $1\leq i\leq M$, $v_i=u_i^nx_i^{-s_i}$. 
\item If all $v_i$ have the form $u_i^{r_i}$, $r_i\geq 0$, then, $$\mu(v)=\sum_{i=1}^{M} r_i\mu_i(u_i)$$
where $ \mu(u_i) =k$  $\Leftrightarrow$ if $x_i \in \tau_k(G) \smallsetminus  \tau_{k+1}(G)$.
\end{itemize}
\end{defn}

\begin{lem} \label{mainlem} The following statements are true.
\begin{enumerate}
\item Each quotient $I^k/I^{k+1}$, $k=0, \ldots n-1$, is spanned by the cosets of those basis elements $v$ of Lemma~\ref{basis:kG:lem}, for which $\mu(v)=k$,
\item The action of $G$ by right multiplication on the series
$$\fra{K}G/I^n  >I^1/I^n > \dots > I^n/I^n = 0,$$
  is nilpotent; that is, for $x\in G$, $u_{j}\in I^k$, and $u_j\not\in I^{k+1}$, $u_{j} x=u_{j} + y$, for some $y \in I^{k+1}$ (to unify notation we put  $I^0 = \fra{R}G$). 
 \item $\Delta_k(G) = \tau_k(G)$ for all $k$, in particular $\Delta_n(G) = \tau_n(G)=1$ and so the homomorphism $\phi_n: G \to Aut(\fra{R}G/I^n)$ is an embedding.\end{enumerate}\end{lem} 
 
The embedding theorem is a direct corollary of Lemma~\ref{mainlem}.

 \begin{thm} Every $\tau$-group $G$ can be embedded into $UT_d(\Z)$, where $\ds d=\sum_{k=0}^{n-1} d_k$, and $d_k$ is the dimension of $I^k/I^{k+1}$ as a $\fra{R}$-vector space.
   
\end{thm}
We refer to the embedding obtained above as the \emph{Jennings embedding} of a $\tau$-group into a unitriangular group.

\subsection{Nickel's embedding}
Now we are going to consider the matrix representation for torsion-free nilpotent groups given by W. Nickel~\cite{N}.

Nickel presents an algorithm for calculating a representation by unitriangular matrices over the integers of finitely-generated torsion-free nilpotent group given by a polycyclic presentation. The algorithm uses polynomials computed by Deep Thought algorithm which describes the multiplication in the given group and is presented by Leedham-Green and Soicher~\cite{LGS}. Nickel in his notes shows the existence of a faithful unitriangular matrix representation as a simple corollary of the fact that the multiplication can be described by polynomials.

 Here we will give a brief description of Nickel's embedding and prove the Theorem \ref{le:Nickel's embedding}. 
\\
Let $G$ be a finitely generated torsion-free nilpotent group and consider the Mal'cev basis $(x_1,x_2, \dots, x_m)$ for $G$ as defined in case of Jennings' embedding.  Each $g \in G$ can be written uniquely as a normal form $g={x_1}^{a_1} \dots {x_m}^{a_m}$ with integers $a_1, \dots a_m$. 
In particular the product of two elements can be written in the same fashion 
$$ {x_1}^{a_1}\dots {x_m}^{a_m} \cdot {x_1}^{b_1} \dots {x_m}^{b_m}= {x_1}^{q_1} \dots {x_m}^{q_m}.$$
The exponents $q_1,\dots, q_m$ are functions of $a_1, \dots ,a_m$ and $b_1, \dots,b_m$. Hall~\cite{hall} showed that these functions are polynomials. We call $q_1, \dots, q_m$ the multiplication polynomials for the Mal'cev basis $(x_1, \dots, x_m)$. 
Deep Thought computes the multiplication polynomials for $(x_1,\dots, x_m)$ from Mal'cev basis $(x_1,\dots, x_m)$. The algorithm presented in Nickel's paper takes as an input Mal'cev basis and corresponding multiplication polynomials and computes a unitriangular matrix representation for $G$ over the integers.

To construct the representation W. Nickel uses the fact that the dual 
$$ (\Q G)^* =\{ f: \Q G \to \Q \: | \: \textit{f is linear} \} $$
can be seen as a G-module, where $G$ acts on  $(\Q G)^*$ as follows: for $g\in G$ and $f\in (\Q G)^*$ define $f^g$ to be the function that maps each $h\in G $ to $f(h\cdot g^{-1})$. Identifying ${x_1}^{a_1}\dots {x_m}^{a_m}$ as $a_1,\dots ,a_m$ and writing $f(a_1,\dots, a_m)$ instead of $f({x_1}^{a_1}\dots {x_m}^{a_m})$ allows us to view $\Q [ a_1,\dots, a_m]$ as a subset of $(\Q G)^*$. The image of $f\in (\Q G)^*$ under $g^{-1}={x_1}^{b_1}\dots {x_m}^{b_m}$ can be described with the help of the polynomials $q_1,\dots, q_m$, for $h={x_1}^{a_1}\dots {x_m}^{a_m}$ we have $hg^{-1}= {x_1}^{q_1}\dots {x_m}^{q_m}$. Therefore applying $g^{-1} \in G$ to a function $f$ amounts to substituting the multiplication polynomials into $f$. If $f$ is itself a polynomial on $G$, then $f(q_1, \dots,  q_m)$ is a polynomial in $a_1, \dots , a_m$ and $b_1, \dots, b_m$.
For proofs of the following two lemmas, refer to ~\cite{N}.
\begin{lem}
Let $f \in \Q[a_1,\dots,a_m]$, then $G$-module of $(\Q G)^*$ is generated by $f$ is finite-dimensional as a $\Q$-vector space.
\end{lem}
The following lemma shows how to construct a finite dimensional faithful $G$-module of $(\Q G)^*$. 
Now we are going to consider the $i$'th coordinate function mapping $t_i: G \to \Z$ which maps ${x_1}^{a_1}\dots {x_m}^{a_m}$ to $a_i$. Note that $t_i$ is well defined because each element of $G$ can be written uniquely in the form ${x_1}^{a_1}\dots {x_m}^{a_m}$.  
\begin{lem}\label{le:faithful module}
The module $M$ of $(\Q G)^*$ generated by $t_1,\dots, t_m$ is a finite dimensional faithful $G$-module.
\end{lem}
As a result of the Lemma \ref{le:faithful module}, we can see that $G$ has a matrix representation for some $n\in\N$. As part of the algorithm Nickel also explains how the unitriangular presentation of $G$ is obtained. For the details of the algorithm, refer to ~\cite{N}.

\section{Distortion of embeddings} \label{main} 
\subsection {Distortion of Jennings' embedding}
Here we discuss the distortion of the Jennings' embedding of a $\tau$-group into a unitriangular group. First we consider the simplest non-trivial example of an embedding of $\tau$-group into a unitriangular group over $\Z$. Indeed we will try to embed $G=UT_3(\Z)$ into a unitriangular group via Jennings' recipe. We shall see that the standard Jennings embedding distorts the image of $G$.  Then we will come up with a fix and prove that the embedding can be made without distortion once we choose a different order on the basis elements of $\fra{R}G$.  
   
\begin{rem} Note that a matrix in $UT_n(\Z)$ belongs to $$\Gamma_l(UT_n(\Z))\setminus \Gamma_{l+1}(UT_n(\Z))$$ if and only if at least one l'th super-diagonal entry is non-zero and all other k'th super-diagonal entries are zero for $k \leq l $. \end{rem}

\begin{exmp}\label{distorted:exmp} Let $$G= \langle x,y,z | [x,y]=z,[x,z]=[y,z] = 1 \rangle \simeq UT_3(\Z).$$
Then the ordered triple $(x,y,z)$ is a Mal'cev basis for $G$,
$$ (1+I^3,u+I^3,v+I^3,u^2+I^3,uv+I^3, v^2+I^3,w+I^3),$$
where  $u=1-x$, $v=1-y$  and $w=1-z$, is a $\fra{R}$-basis of $\fra{R}G/I^3 $. In this example we will show that the embedding $\phi:G\to UT_7(\Z)$ induced by the right action of $G$ on the ordered basis above is actually distorted.
 
 The analysis of the distortion will be easy once we sorted out the embedding.
The following are obvious.
\begin{itemize}
\item $\tau_2(G) = \Gamma_2 = \langle z \rangle$
\item $G/ \tau_2(G) = G/ \Gamma_2 = \langle x\Gamma_2, y\Gamma_2\rangle $
\item the ordered triple $(x,y,z)$ is a Mal'cev basis for $G$.
\end{itemize}
Now consider $\fra{R}G$ and let $u=1-x$, $v=1-y$  and $w=1-z$, so we have $\mu(u)=1$, $\mu(v)=1$, $\mu(w)=2$.
It is not hard to verify that 
$$ (1+I^3,u+I^3,v+I^3,u^2+I^3,uv+I^3, v^2+I^3,w+I^3)$$
is a basis for $\fra{R}G/I^3$. Indeed 
\begin{align*}
I^0/I^1 &=  span\{1+I^1\}\\
I^1/I^2 &=  span\{u+I^2,v+I^2\}\\
I^2/I^3 &=  span\{u^2+I^3 , uv+I^3, v^2+I^3, w+I^3\}.\end{align*}
for further calculations we need to express $vu$ in terms of the above basis elements. Firstly we need to express $z^{-1}$ in terms of the basis presented in Lemma~\ref{basis:kG:lem}. Repeated applications of the identity $$w^nz^{-r}=w^{n+1}z^{-r}+u^nx^{-r+1}, \: r\geq 0, n\geq 0$$ will result in 
$$z^{-1}=w^3z^{-1}+w^2+w+1,$$
where all the elements on the right hand side are basis elements. Now 
\begin{equation}\label{eq2}
\begin{split}
vu&=(1-y)(1-x)=1-y-x+yx\\
&=1-y-x+xy[y,x]=1-y-x+xyz^{-1}\\
&=1-y-x+xy(w^3z^{-1}+w^2+w+1)\\
&=v-(1-u)+(1-u)(1-v)(w^3z^{-1}+w^2+w+1)\\
&\equiv uv + w \quad (mod ~I^3).
\end{split}  \end{equation}

Now let us find matrix representations of the right actions of $x$ and $y$ on the above ordered-basis. All congruences, $\equiv$, are modulo $I^3$. 
\begin {align*} 
1 x&= 1(1-u)=1 - u \\
 u x&= u(1-u)= u-u^2 \\
 v x &=v (1-u) = v-vu \equiv v-uv -w\\ 
 u^2x&= u^2 (1-u)= u^2-u^3\equiv u^2\\
 uvx&=uv(1-u)=uv-uvu\equiv uv \\
 v^2x&=v^2(1-u)=v^2-v^2u\equiv v^2\\
 w x&= w(1-u)= w-wu \equiv w 
    \end{align*}
so, the corresponding matrix:
$$\left[\begin{array}{ccccccc} 1&-1&0&0&0&0&0 \\ 0&1&0&-1&0&0&0 \\ 0&0&1&0&-1&0&-1 \\ 0&0&0&1&0&0&0 \\ 0&0&0&0&1&0&0 \\0&0&0&0&0&1&0 \\0&0&0&0&0&0&1 \end{array}\right]$$
Let $K= UT_7(\Z)$ and $\phi: G \to K$ denote our embedding. 

Similarly, for the action of $y$ we have  
\begin {align*} 
1y&= 1(1-v)=1-v \\
 u y&= u(1-v)= u-uv\\
  vy &=v (1-v) = v-v^2 \\
    u^2y&= u^2 (1-v)= u^2-u^2v  \equiv u^2 \\
 uvy&=uv(1-v)=uv-uv^2\equiv uv \\
 v^2y&=v^2(1-v)=v^2-v^3\equiv v^2\\
wy&=w(1-v)=w-wv\equiv w
\end{align*}
so, the corresponding matrix is:
$$\left[\begin{array}{cccccccc} 1&0&-1&0&0&0&0 \\ 0&1&0&0&-1&0&0 \\ 0&0&1&0&0&-1&0 \\ 0&0&0&1&0&0&0 \\ 0&0&0&0&1&0&0 \\0&0&0&0&0&1&0 \\0&0&0&0&0&0&1 \end{array}\right]$$

The matrix representation of $z$ is found similarly and it is the matrix with -1 as the (1,7)'th entry, 1 down the diagonal and the rest of entries 0.
 
Studying the matrix $\phi(y)$  we notice that $\phi(y)\in \Gamma_2(K)\setminus \Gamma_3(K)$. This means that $\nu_G(y)=\nu_{\phi(G)}(\phi(y))=1$ but $\nu_K(\phi(y))=2$. This together with Theorem~\ref{utnZR} imply that $\phi(G)$ is distorted in $K$.  

\end{exmp}

Despite the observation made above we can come up with an easy fix for the problem by simply imposing a different order on the basis of $\fra{R}G$. 

\begin{exmp} Let's put a different ordering on the basis of $\fra{R}G/I^3$ from Example~\ref{distorted:exmp}, namely
$$ (1+I^3, v+I^3, w+I^3, u+I^3, u^2+I^3, uv+I^3, v^2+I^3).$$

Now we shall use the calculations in Example~\ref{distorted:exmp} to find the matrices of the actions of $x$, $y$ and $z$ by right multiplication on the above basis and denote them by $\psi(x)$, $\psi(y)$ and $\psi(z)$ respectively. Note we use $\psi$ to denote the embedding under the new ordering on the basis of $\fra{R}G/ I^3$.

$$\psi(x)=\left[\begin{array}{ccccccc} 1&0&0&-1&0&0&0 \\ 0&1&1&0&0&-1&0 \\ 0&0&1&0&0&0&0 \\ 0&0&0&1&-1&0&0 \\ 0&0&0&0&1&0&0 \\0&0&0&0&0&1&0 \\0&0&0&0&0&0&1 \end{array}\right]$$

$$\psi(y)=\left[\begin{array}{cccccccc} 1&-1&0&0&0&0&0 \\ 0&1&0&0&0&-1&0 \\ 0&0&1&0&0&0&0 \\ 0&0&0&1&0&-1&0\\ 0&0&0&0&1&0&0 \\0&0&0&0&0&1&0 \\0&0&0&0&0&0&1 \end{array}\right]$$

$$\psi(z)=\left[\begin{array}{ccccccc} 1&0&-1&0&0&0&0 \\ 0&1&0&0&0&0&0 \\ 0&0&1&0&0&0&0 \\ 0&0&0&1&0&0&0 \\ 0&0&0&0&1&0&0 \\0&0&0&0&0&1&0 \\0&0&0&0&0&0&1 \end{array}\right]$$

Obviously the image of $\psi$ sits inside $UT_7(\Z)$ and moreover $\nu_G(x)=\nu_G(y)=\nu_{\psi(G)}(x)=\nu_{\psi(G)}(y)=1$ and $\nu_K(\psi(x))=\nu_K(\psi(y))=1$, also $\nu_G(z)=\nu_{\psi(G)}(\psi(z))=\nu_K(\psi(z))=2$ and so by  Theorem~\ref{utnZR}, $\psi(G)$ is undistorted in $K$. Now we will see in the following theorem that we can generalize this for group of unitriangular matrices $UT_n(\Z)$ for $n \geq 3$. \end{exmp} 

\begin{thm} \label{UTNundistorted:thm} Under Jennings' embedding $G=UT_m(\Z)$ can be seen as an undistorted subgroup of $UT_d(\Z)$ for $d >m$. 
\end{thm}

\begin{proof} By Theorem~\ref{utnZR} we need to find an embedding $\phi: G \rightarrow UT_d(\Z)=K$ such that if $\nu_G(g) =\ell $ for $g\in G\setminus{\{1\}}$, then $\nu_K( \phi(g))=\ell$.

We will first obtain a Mal'cev basis for $G=UT_m(\Z)$. We will denote each generator of $G$ in terms of $s_{ij}$ which represents $m\times m$ upper unitriangular matrix with $ij$'th entry $1$ as described earlier.

So
 $$ G= \langle x_1, x_2, x_3,\dots, x_{M-1}, x_M \rangle $$
where
	\begin{align*}
		x_k &= s_{ji} &\text{for }&k = j + \sum_{\ell = 1}^{i-j-1} (m-\ell).
	\end{align*}
that is;
 $$
 \begin{pmatrix}
1  &x_1    & x_m      &   \cdots    &       x_{M}  \\
   & 1      & x_2 & x_{m+1} &          \\
   &        & \ddots & \ddots &  \vdots   \\
          & 0 &  & 1  & \hspace{-2mm}x_{m-1} \\
   &        &        &            & 1
\end{pmatrix}.
$$ 

Note that $s_{ij} \in \Gamma_k(G)$ if $j-i=k$, therefore the following is a Mal'cev basis for $G$ obtained from lower central series of $G$, which is also considered to be the standard one. 
$$(x_1,x_2, \ldots x_{m-1}, \ldots x_M)$$ where $M=\frac{m(m-1)}{2}$.

However we will use a slightly different Mal'cev basis, and we leave it to the reader to check that it is indeed a Mal'cev basis for $G$. So if we reorder elements such that $s_{ij}$ shows up on the left side of $s_{k \ell}$ if and only if $j < \ell$ or $j=\ell$ and $i > k$, that is according to the scheme
	$$
	\begin{pmatrix}
	1  &x_1& x_3   & x_6 & \cdots & x_{n} \\
	& 1 & x_2   & x_{5}& & x_{n-1}\\
	&   &   1   & x_4 & &\\
	&   &       & 1      &  &\smash{\vdots}  \\
	& 0 &       &    &  \ddots & \\ 
	&   &       &      &  & 1
	\end{pmatrix}
	$$

We now set $u_i= 1-x_i$. The order imposes a ``natural" (left or right) lexicographical order on the basis of $\fra{R}G/ I^{m}$, where $c$ is the nilpotency class of $G$ and $m=c+1$. We will perturb the order  as follows
$$(1,u_{(m-1)m},u_{(m-2)m},\dots, u_{1m},u_{(m-2)(m-1)},u_{(m-3)(m-1)},\dots,u_{23},u_{13},u_{12},\ovr{u}),$$
where $u_{ij}=1-s_{ij}$ and $\ovr{u}$ represents the final segment of the basis in its natural order. Note that for simplicity we preferred to use $s_{ij}$ notation rather than $x_k$ for the generators of $G$. 
 
Now we claim that

\begin{enumerate}
\item[(a)] The right action of $G$ on this ordered basis ($mod~I^m$) is represented by a unitriangular matrix.
\item[(b)] $\nu_K(\phi(x_i))=\ell$ for $\nu_G(x_i)=\ell$ for all $i=1, \ldots ,M$, so the weights of elements are preserved under the embedding.
\end{enumerate} 
To prove (a) we will show that every elements $u_{ij}=1-s_{ij}$ in the initial segment  
$$(1,u_{(m-1)m},u_{(m-2)(m)},\dots, u_{1m},u_{(m-2)(m-1)},u_{(m-3)(m-1)},\dots,u_{23},u_{13},u_{12})$$ of the basis has an upper triangular image, because we know that otherwise the action is nilpotent by Jennings' theorem on the rest of the basis elements. Note that congruences are modulo $I^{m}$.

\begin{equation}\label{eq3}
\begin{split}
u_{j\ell}s_{ik}& = u_{j\ell} (1-u_{ik})\\&= u_{j\ell}-u_{j\ell}u_{ik}\\& \equiv u_{j\ell}-u_{ik}u_{j\ell}.
\end{split}
\end{equation}
for  $ 1\leq i,j,k,\ell \leq m$ and $j<\ell$, $i < k$, $i\neq \ell$ $k \neq j$ and with the assumption that $s_{ik}$ is on the left side of $s_{j\ell}$. Since $s_{ik}$ commutes with $s_{j \ell}$, no new single element gets introduced. Moreover, the element $u_{ik}u_{j\ell}$ comes after all the elements in the initial segment of the basis.

\begin{equation}\label{eq4}
\begin{split}
u_{ji}s_{ik}& = u_{ji} (1-u_{ik})\\&= u_{ji}-u_{ji}u_{ik} 
\end{split}
\end{equation}
for $1 \leq j <i < k \leq m$, since in the Mal'cev basis we picked, the element $s_{ji}$ comes before $s_{ik}$. 
As before, in the reordered basis of $\fra{R}G/I^m$ $u_{ji}u_{ik}$ comes after all single elements.

And finally,
\begin{equation}\label{eq5}
\begin{split}
u_{ji}s_{kj} & = u_{ji} (1-u_{kj}) \\ &= u_{ji}-u_{ji} u_{kj}\\ &\equiv u_{ji}-u_{kj}u_{ji}+u_{ki}.
\end{split}
\end{equation}
for  $ 1\leq k < j < i \leq m$,  since $[s_{ji},s_{kj}]=-s_{ki}$. Note that
$u_{ki}$ comes after $u_{ji}$ and $u_{kj}u_{ji}$ comes after all the elements in the initial segment.  
Also notice that, when the action of $s_{ij}$'s are applied to elements in the product form such as $u_{kl}u_{rs}$, no single elements shows up due to the Mal'cev basis we picked as above and so unitriangularity of the action is preserved.
This way we proved (a).

To prove (b) we need to show that weight of each element is preserved. First notice that $s_{ij}$ belongs to $\Gamma_{\ell}(G)$ if $j-i=\ell$, so to show that weights of elements are preserved, we need to show that the weight of the image of $s_{ij}$ in $K$ is also $\ell$. 
Let's have a look at the action of $s_{ij}$ for $j-i=\ell$ over the element $u_{jk}$ for $i < j < k $,
\begin{equation}\label{eq6}
\begin{split}
u_{jk}s_{ij} & =u_{jk}(1-u_{ij})\\&
= u_{jk}-u_{jk}u_{ij}\\&
\equiv u_{jk}-u_{ij}u_{jk}-u_{ik}.
\end{split}
\end{equation}
Now note that in the basis the distance between $u_{jk}$ and $u_{ik}$ is $\ell$, which proves that for each $s_{ij}$, the matrix representation $\phi(s_{ij})$, of $s_{ij}$ has at least one non-zero ${\ell}^{th}$ super-diagonal entry. This proves (b) and finishes the proof of the theorem. 
\end{proof}
\begin{thm}\label{HeisenbergJ:thm}
For $n\geq 2$ the image of any $(2n+1)$-dimensional Heisenberg group is always distorted under Jennings' Embedding. 
\end{thm}
\begin{proof}We first like to provide a finite presentation for the generalized Heisenberg group.
\\
Now we first give a finite presentation to the generalized Heisenberg group as follows:
$$G= \langle x_1, \dots x_{2n+1} | \:R\: \rangle $$
for $s_{ij}$ described as above and $x_i$'s are defined as follows $$x_i = \begin{cases} s_{1(i+1)} &\mbox{for}\:1 \leq i\leq n \\  s_{(i-n+1)(n+2)}&\mbox{for}\:n+1 \leq i < 2n+1 \\ s_{1(n+2)} &\mbox{for}\:i=2n+1 \end{cases}$$ and the set $R$ is determined by using the fact that 
$$s_{ij}^{-1}=\left(s_{ij}(1)\right)^{-1}= s_{ij}(-1),$$  $$[s_{ij}, s_{jk}]=s_{ik}, \: \text{and}\: [s_{ji}, s_{kj}]=s_{ik}(-1) \: \text{for} \: i<j<k.$$ 
So,
$$R= \{ [x_i, x_{n+i}]=x_{2n+1}\:\text{for} \:1 \leq i\leq m \:\text{and all other pairs of}\:x_j\: \text{commute}\}.$$ 
\\
Note that $(x_1,x_2,\dots, x_{2n+1})$ is a Mal'cev Basis for generalized Heisenberg group given as above. 
Notice that for $n \geq 2$ to prove the embedding remains distorted for all possible orderings of the basis elements of $\fra{R}G/I^3$, we will show that it is impossible to preserve the weight of $x_{2n+1}$ and keep the upper-triangularity. Also note that we have that $\Gamma_2(G)= Z(G)= \langle x_{2n+1} \rangle$, so $\nu(x_{2n+1})=2$. 

Now let $u_i=1-x_i$. First note that to preserve the upper triangularity, the elements $u_{n+i}$ must show up on the left of $u_{2n+1}$, which enforces the weight of $x_{2n+1}$ to be $n+1$ and $n\geq 2$. Hence the embedding remains distorted for all possible orderings of the basis elements of $\fra{R}G/I^3$.
\end{proof}

\begin{thm}\label{distorted1:thm} Jennings's embedding can not be made without distortion for $\tau$-groups with rank of $\Gamma_c$ greater or equal to 2. 
\end{thm}
Before we give the proof of the Theorem~\ref{distorted1:thm} we will first have a look at an example. 
Let group $G$ be given with the following presentation 

\begin{align*}
G= \langle y_1,y_2,y_3,y_4,y_5 \: | \:  & [y_1,y_2]=y_3,\:[y_3,y_1]=y_4,\:[y_3,y_2]=y_5, \\& [y_3,y_4]=[y_3,y_5]=[y_4,y_5]=1\rangle
\end{align*}
Note that $G$ is simply the free nilpotent group of rank $2$ and of class $3$, where $(y_1,y_2,y_3,y_4,y_5)$ is a Mal'cev basis for $G$ and the weights of elements are as follows:
$$\nu(y_1)=\nu(y_2)=1, \nu(y_3)=2, \nu(y_4)=\nu(y_5)=3.$$
Under Jenning's embedding, we like to show that $G$ sits as a distorted subgroup of the unitriangular group of order $15$.

Now let $u_i=1-x_i$ for $1\leq i\leq 5$, so we can list the basis elements of $\fra{R}G/I^4$ in lexicographical order as follows:
$$(1,u_1,u_2,u_1^2,u_1u_2,u_2^2,u_3,u_1^3,u_1^2u_2,u_1u_2^2,u_1u_3,u_2^3,u_2u_3,u_4,u_5)$$

Since 
\begin{align*}
I^0/I^1 &=  span\{1+I^1\}\\
I^1/I^2 &=  span\{u_1+I^2, u_2+I^2\}\\
I^2/I^3 &=  span\{u_1^2+I^3, u_1u_2+I^3,u_2^2+I^3,u_3+I^4\}\\
I^3/I^4 &=  span\{u_1^3+I^4, u_1^2u_2+I^4, u_1u_2^2+I^4, u_1u_3+I^4, u_2^3+I^4,\\& u_2u_3+I^4,u_4+I^4, u_5+I^4\}.
\end{align*}
Note that $\{y_4,y_5\}$ generates $\Gamma_3(G)= Z(G)$.
Now we like to have a look at the actions of $y_4$ and $y_5$ over the basis elements of $\fra{R}G/I^4$.
\begin{equation}\label{eq7}
\begin{split}
& 1y_4=1(1-u_4)=1-u_4 \\& 
u_1y_4=u_1(1-u_4)\equiv u_1\\& u_2y_4=u_2(1-u_4)\equiv u_2 \\&
u_1^2 y_4 = u_1^2(1-u_4)= u_1^2-u_1^2u_4 \equiv u_1^2 \\& 
u_1u_2 y_4 =u_1u_2 (1-u_4)=u_1u_2-u_1u_2u_4 \equiv u_1u_2\\& 
u_2^2 y_4 = u_2^2 (1-u_4)= u_2^2-u_2^2u_4 \equiv u_2^2\\&
u_3y_4 = u_3 (1-u_4)=u_3-u_3u_4\equiv u_3\\&
u_1^3y_4= u_1^3(1-u_4)= u_1^3-u_1^3u_4 \equiv u_1^3 \\& \vdots \\& 
u_2u_3 y_4 = u_2u_3 (1-u_4) = u_2u_3-u_2u_3u_4 \equiv u_2u_3 \\& 
u_4y_4=u_4(1-u_4)\equiv u_4\\& u_5y_4=u_5(1-u_4)=u_5-u_5u_4\equiv u_5 \\&
\end{split}
\end{equation}

 Since $y_4$ commutes with all other elements of $G$, the only way the image of this element to have a non-zero $3^{rd}$ super-diagonal entry is placing $u_4$ exactly $3$ elements away from the identity element $1$. Notice that the same is true for $u_5$ since $y_5$ also belongs to $\Gamma_3(G)$ and its action over the basis elements of $\fra{R}G/I^4$ is the same as $y_4$. Therefore such an ordering is impossible and hence $G$ sits as a distorted subgroup of the unitriangular group $UT_{15}(\Z)$.
\\
\\
\textit{Proof of Theorem~\ref{distorted1:thm}.} As it can be seen in the above example if a $\tau$-group $G$ has $\Gamma_c(G)$ with rank greater or equal to 2, under Jennings' embedding this group sits as a distorted subgroup of a unitriangular group.  It is simply because all the elements that generates the subgroup $\Gamma_c(G)$ must be located in the same place. Considering identity is the first element in the order, generators of $\Gamma_c(G)$ need to be placed in the $c$'th location after the identity to  make sure the image of elements of weight $c$ remains unchanged. Clearly, this is impossible. 
\begin{cor} 
Under Jennings' Embedding, the image of $F(k,c)$, free nilpotent group of rank $k$ and class $c$ can not be made without distortion for $k\geq 2$ and $c\geq 3$.
\end{cor}

\subsection{Distortion of Nickel's Embedding}
\indent 
\begin{thm}\label{UTNundistortedN:thm} Under Nickel's embedding $G=UT_m(\Z)$ can be seen as an undistorted subgroup of $UT_d(\Z)$ for $d=n +1$, where $n=\frac{m(m-1)}{2}$. 
\end{thm}
\begin{proof}
As in Jennings' we will use a non-standard Mal'cev basis such that $s_{ij}$ shows up on the left of $s_{k\ell}$ if and only if $j < \ell$ or $j=\ell$ and $i > k$, that is; according to the scheme
	$$
	\begin{pmatrix}
	1  &x_1& x_3   & x_6 & \cdots & a_{n} \\
	& 1 & x_2   & x_{5}& & x_{n-1}\\
	&   &   1   & x_4 & &\\
	&   &       & 1      &  &\smash{\vdots}  \\
	& 0 &       &    &  \ddots & \\ 
	&   &       &      &  & 1
	\end{pmatrix}
	$$
\\
Now we like to show that if the following order is selected for the basis elements of $G$-module, the embedding becomes undistorted. For simplicity we will use $s_{ij}$ instead of $x_k$ for the generators of $G$, and so similarly we will use $t_{ij}$ notation instead of $t_k$ for the coordinate functions.

$$ (t_{12},t_{13},t_{23}, \dots, t_{1j},t_{2j},\dots, t_{(j-1)j}, \dots t_{1m},t_{2m}, \dots t_{(m-1)m},1)$$
where $2 \leq j\leq m$.	
Also note that the following is the Mal'cev basis we picked for $G$ in $s_{ij}$ notation.
$$( s_{12},s_{23},s_{13},\dots s_{(j-1)j}, s_{(j-2)j},\dots, s_{1j}, \dots, s_{(m-1)m},s_{(m-2)m},\dots,s_{1m})$$
where $2 \leq j\leq m$

Now we claim that

\begin{enumerate}
\item[(a)]  The right action of $G$ on this ordered basis of $G$-module is represented by a unitriangular matrix,
\item[(b)] $\nu_K(\phi(s_{kl}))=\ell$ for $\nu_G(s_{kl})=\ell$ for all $1 \leq \ell \leq m-1$ and $K= UT_{n+1}(\Z)$, so the weights of elements are preserved under the embedding.
\end{enumerate} 

We show both (a) and (b) holds by simply looking at the action of $s_{kl}$ over each $t_{ij}$, where $l-k=\ell$. 
So let's have a look at the following multiplication. 
\begin{align*}
s_{12}^{a_{12}}s_{23}^{a_{23}}s_{13}^{a_{13}}\dots s_{(j-1)j}^{a_{(j-1)j}} s_{(j-2)j}^{a_{(j-2)j}}\dots s_{1j}^{a_{1j}}\dots s_{(m-1)m}^{a_{(m-1)m}}s_{(m-2)m}^{a_{(m-1)m}}\dots s_{1m}^{a_{1m}}s_{kl}^{-b}
\end{align*}
Now note that $s_{kl}$ doesn't commute with only the elements of the form $s_{lr}$ and $s_{pk}$, where $p < k < l < r$, and $s_{pk}$ is on the left side of $s_{kl}$ in the Mal'cev basis
\begin{align*}
s_{12}^{a_{12}}& s_{23}^{a_{23}}s_{13}^{a_{13}} \dots s_{pk}^{a_{pk}} \dots s_{kl}^{a_{kl}}\dots s_{lr}^{a_{lr}}s_{kl}^{-b}\dots  s_{(m-1)m}^{a_{(m-1)m}}s_{(m-2)m}^{a_{(m-1)m}}\dots s_{1m}^{a_{1m}}\\
&=s_{12}^{a_{12}}s_{23}^{a_{23}}s_{13}^{a_{13}} \dots  s_{pk}^{a_{pk}} \dots  s_{kl}^{a_{kl}-b}\dots s_{kr}^{a_{kr}-ba_{lr}}\dots  s_{(m-1)m}^{a_{(m-1)m}}s_{(m-2)m}^{a_{(m-1)m}}\dots s_{1m}^{a_{1m}}
\end{align*}

\begin{align*}
t_{ij}^{s_{kl}^{-b}}&(s_{12}^{a_{12}}s_{23}^{a_{23}}s_{13}^{a_{13}}\dots s_{(j-1)j}^{a_{(j-1)j}} s_{(j-2)j}^{a_{(j-2)j}}\dots s_{1j}^{a_{1j}}\dots s_{(m-1)m}^{a_{(m-1)m}}s_{(m-2)m}^{a_{(m-1)m}}\dots s_{1m}^{a_{1m}}) \\&=
t_{ij}(s_{12}^{a_{12}}s_{23}^{a_{23}}s_{13}^{a_{13}} \dots  s_{pk}^{a_{pk}} \dots  s_{kl}^{a_{kl}-b}\dots s_{kr}^{a_{kr}-ba_{lr}}\dots  s_{(m-1)m}^{a_{(m-1)m}}s_{(m-2)m}^{a_{(m-1)m}}\dots s_{1m}^{a_{1m}})
\end{align*}
 so we have the following;
 \begin{align*}
  t_{kl}(\bar{s}^{\bar{a}} s_{kl}^{-b}) & = a_{kl}-b=t_{kl}-b \\
 t_{kr}(\bar{s}^{\bar{a}} s_{kl}^{-b})&= a_{kr}-ba_{lr}=t_{kr}-bt_{lr} \\
 t_{ij}(\bar{s}^{\bar{a}} s_{kl}^{-b})&= a_{ij}=t_{ij},\: \textit{for any}\:i \neq k\:
  \end{align*}
 where
 $$\bar{s}^{\bar{a}}=s_{12}^{a_{12}}s_{23}^{a_{23}}s_{13}^{a_{13}}\dots s_{(j-1)j}^{a_{(j-1)j}} s_{(j-2)j}^{a_{(j-2)j}}\dots s_{1j}^{a_{1j}}\dots s_{(m-1)m}^{a_{(m-1)m}}s_{(m-2)m}^{a_{(m-1)m}}\dots s_{1m}^{a_{1m}}$$
 
	First notice that with the non-standard Mal'cev basis, $UT_m(\Z)$ is embedded into $UT_{n+1}(\Z)$ where $n=\frac{m(m-1)}{2}$, since the actions of the group elements over coordinate functions only creates functions that are already linear combinations of those coordinate functions. So we only need to add the identity element to the basis of $G$-module.

We can see that with the order taken on the basis of $G$-module the action is represented by a unitriangular matrix, since $t_{lr}$ is on the right side of $t_{kr}$. Also the weight of image of $s_{kl}$ is also $\ell$ since the number of basis elements between $t_{kr}$ and $t_{lr}$ is $\ell-1$, so we proved both (a) and (b).  
\end{proof}

\begin{thm}\label{le:Nickel's embedding}
For $n\geq 2$ the image of any $(2n+1)$-dimensional Heisenberg group under Nickel's Embedding is always distorted.
\end{thm}
\begin{proof} First note that, by Theorem \ref{UTNundistortedN:thm}, for $3$-dimensional Heisenberg group the statement is not true.
\\
Now let's show that under Nickel's embedding for $n \geq 2$, the image of any $(2n+1)$-dimensional Heisenberg group is always distorted. 
In this proof we will use the same presentation for the generalized Heisenberg group as described in the study of distortion of Jennings' embedding.
\\ 
For simplicity we will write $ x_1^{a_1}x_2^{a_2} \dots x_{2n+1}^{a_{2n+1}}=\bar{x}^{\bar{a}}$. We will now show that the following is true. \\
For $ 1\leq j \leq n $,
$$t_i^{{x_j}^k}(\bar{x}^{\bar{a}})= \begin{cases}{a_{j}-k} &\mbox{for}\: i=j \\ a_i &\mbox{for}\:i\neq j\:\text{and}\: i\neq 2n+1 \\ a_{2n+1}+ka_{n+j} &\mbox{for}\:i=2n+1 \end{cases} $$
For $ n+1 \leq j \leq 2n+1 $,

$$t_i^{{x_j}^k}(\bar{x}^{\bar{a}})= \begin{cases}{a_{j}-k} &\mbox{for}\: i=j \\ a_i &\mbox{for}\:i\neq j \end{cases} $$

In order to see this, we only need to have a look at
$$ \bar{x}^{\bar{a}}{x_j}^{-k}=x_1^{a_1}x_2^{a_2}\dots x_{2n+1}^{a_{2n+1}}{x_j}^{-k}.$$

For $ 1\leq j \leq n $, since $s_{1(j+1)}=x_j$ commutes with all except $s_{(j+1)(n+2)}=x_{n+j}$ and also $[s_{1(j+1)}^{-k},s_{(j+1)(n+2)}^{-a_{n+j}}]=s_{1(n+2)}^{ka_{n+j}}=x_{2n+1}^{ka_{n+1}}$, we have   
\begin{align*}
\bar{x}^{\bar{a}}{x_j}^{-k}&= s_{12}^{y_1}s_{13}^{a_2} \dots s_{1(j+1)}^{a_j} \dots s_{2(n+2)}^{a_{n+1}} \dots s_{1(n+2)}^{a_{2n+1}}{s_{1(j+1)}}^{-k} \\
&= s_{12}^{a_1}s_{13}^{a_2} \dots s_{1(j+1)}^{a_j} \dots s_{1(j+1)}^{-k}{s_{1(j+1)}}^ks_{(j+1)(n+2)}^{a_{n+j}}s_{1(j+1)}^{-k}
s_{(j+1)(n+2)}^{-a_{n+j}}\\& s_{(j+1)(n+2)}^{a_{n+j}} \dots s_{2(n+2)}^{a_{n+1}} \dots s_{1(n+2)}^{a_{2n+1}} \\
&=s_{12}^{a_1}s_{13}^{a_2} \dots s_{1(j+1)}^{a_j-k} \dots [s_{1(j+1)}^{-k},s_{(j+1)(n+2)}^{-a_{n+j}}]\dots s_{2(n+2)}^{a_{n+1}} \dots s_{1(n+2)}^{a_{2n+1}} 
\\&= s_{12}^{a_1}s_{13}^{a_2} \dots s_{1(j+1)}^{a_j-k} \dots s_{1(n+2)}^{ka_{n+j}}\dots s_{2(n+2)}^{a_{n+1}} \dots s_{1(n+2)}^{a_{2n+1}} 
\\&=s_{12}^{a_1}s_{13}^{a_2} \dots s_{1(j+1)}^{a_j-k} \dots s_{2(n+2)}^{a_{n+1}} \dots s_{1(n+2)}^{a_{2n+1}+ka_{n+j}}.
\end{align*}
For $ n+1\leq j \leq 2n$,
\\
Since $s_{(j-n+1)(n+2)}^{a_j}$ for $ n+1\leq j \leq 2n $ commutes with all other elements.
\begin{align*}
\bar{x}^{\bar{a}}{x_j}^{-k}&= s_{12}^{a_1}s_{13}^{a_2} \dots s_{1(n+1)}^{a_n} \dots s_{1(j+1)}^{a_j} \dots s_{1(n+2)}^{a_{2n+1}}{s_{(j-n+1)(n+2)}}^{-k} \\&=s_{12}^{a_1}s_{13}^{a_2} \dots s_{1(n+1)}^{a_n} \dots s_{(j-n+1)(j+1)}^{a_j-k} \dots s_{1(n+2)}^{a_{2n+1}}\end{align*}
Finally for $j=2n+1$, since $x_{2n+1}=s_{1(n+1)}$ is the last element in $\bar{x}^{\bar{a}}$, we have 
$$\bar{x}^{\bar{a}}{x_{2n+1}}^{-k}= s_{12}^{a_1}s_{13}^{a_2} \dots s_{1(n+1)}^{a_n}  \dots s_{1(n+2)}^{a_{2n+1}-k}$$

In conclusion, because of the fact that all new polynomials are linear combination of $t_i$'s, we have $\{t_1,t_2, \dots, t_{2n+1}, 1\}$ as the $\Q$-basis for the $G$-module. Hence we obtain the embedding $\phi: G \to UT_{2n+2}(\Z)$.  Notice that for any ordering on $\{t_1,t_2, \dots, t_{2n+1}, 1\}$ to result in a unitriangular form, all $t_i$'s for $n<i<2n+1$ must appear after the element $t_{2n+1}$ and the identity 1, which enforces the weight of the image of element $x_{2n+1}$ be at least $n+1$ for $n\geq 2$. So we have $\nu_G(x_{2n+1})=\nu_{\phi(G)}( \phi(x_{2n+1}))=2$ and $\nu_K(\phi(x_{2n+1})) \geq n+1 \geq 3$ since $n\geq2$. 
Therefore under the Nickel's embedding for $n\geq 2$ the image of $(2n+1)$-dimensional Heisenberg group in the unitriangular group is always distorted.
\end{proof}
\begin{exmp}
Let us take a look at the $5$-dim Heisenberg group $G$.
\begin{align*} G= \langle x_1,x_2,x_3,x_4,x_5 \mid & [x_1,x_3]=[x_2,x_4]=x_5,\\& [x_1,x_5]=[x_2,x_5]=[x_3,x_5]=[x_4,x_5]=1 \rangle \end{align*}
Notice that $\Gamma_1(G)=\langle x_1,x_2,x_3,x_4,x_5 \rangle$,  $\Gamma_2(G)= Z(G)= \langle x_5 \rangle$ and $\Gamma_3(G)=1$
Also we can see the elements of $G$ as follows:
$$x_1=\left[\begin{array}{cccc} 1&1&0&0 \\ 0&1&0&0 \\ 0&0&1&0\\ 0&0&0&1  \end{array}\right], x_2=\left[\begin{array}{cccc} 1&0&1&0 \\ 0&1&0&0 \\ 0&0&1&0\\ 0&0&0&1  \end{array}\right], x_3=\left[\begin{array}{cccc} 1&0&0&0 \\ 0&1&0&1 \\ 0&0&1&0\\ 0&0&0&1  \end{array}\right]$$ $$x_4=\left[\begin{array}{cccc} 1&0&0&0 \\ 0&1&0&0 \\ 0&0&1&1\\ 0&0&0&1  \end{array}\right], x_5=\left[\begin{array}{cccc} 1&0&0&1 \\ 0&1&0&0 \\ 0&0&1&0\\ 0&0&0&1  \end{array}\right]$$

We can also take the following as the polycylic series for $G$. 
\begin{align*} G&= G_1=\langle x_1,x_2,x_3,x_4,x_5 \rangle  \geq G_2= \langle x_2,x_3,x_4,x_5 \rangle  \geq \\& G_3=\langle x_3,x_4,x_5 \rangle \geq G_4= \langle x_4,x_5 \rangle \geq G_5= \langle x_5 \rangle  \geq G_6=\langle 1 \rangle \end{align*}
In order to find the image of $G$ under Nickel's embedding $\phi$, we need to have a look at the images of the following products under the coordinate functions.

\begin{align*} x_1^{a_1}x_2^{a_2}x_3^{a_3}x_4^{a_4}x_5^{a_5}x_1^{-k} & = x_1^{a_1}x_2^{a_2}x_3^{a_3}x_4^{a_4}x_1^{-k}x_4^{a_4}x_5^{a_5} \\ & = x_1^{a_1}x_2^{a_2}x_1^{-k}x_1^{k} x_3^{a_3}x_1^{-1}x_3^{-a_3}x_3^{a_3}x_4^{a_4}x_5^{a_5} \\ & = x_1^{a_1-k}x_2^{a_2}[x_1^{-k},x_3^{-a_3}]x_3^{a_3}x_4^{a_4}x_5^{a_5} \\ & = x_1^{a_1-k}x_2^{a_2}x_5^{ka_3}x_3^{a_3}x_4^{a_4}x_5^{a_5}\\ & = x_1^{a_1-k}x_2^{a_2}x_3^{a_3}x_4^{a_4}x_5^{a_5+ka_3}\end{align*}
We leave it to the reader to check all other $x_1^{a_1}x_2^{a_2}x_3^{a_3}x_4^{a_4}x_5^{a_5}x_i^{-k}$, for $1< i\leq 5$. 

We can see that all the exponents in the product after the multiplication by ${x_i}^{-k}$ are $\Q$-linear combinations of $t_i$'s, so we must have $\{t_1, \dots, t_5,1\}$ as the $\Q$-basis for the $G$-module. Now notice that in order to have a unitriangular representation, $t_4$ and $t_3$ must be placed after $t_5$ and in such an ordering only the element $x_i$ that corresponds to the one next to $t_5$ will have an undistorted image, whereas the other element will have a distorded image. If, for example, $t_3$ is placed right after $t_5$, $x_3$ will have an undistorted image since it will have a non-zero super-diagonal and so it will be in $\Gamma_1(UT_6(\Z)) \setminus \Gamma_2(UT_6(\Z))$, however $x_4$ will have an image in $\Gamma_2(UT_6(\Z))$, that is; $\nu_{G}(x_4)=\nu_{\phi(G)}(\phi(x_4))=1$ and $\nu_{UT_6(\Z)}(\phi(x_4))=2$, and so the image of $G$ is distorted in $UT_6(\Z)$. 
\end{exmp}
\begin{thm}\label{distorted2:thm} Similar to Jennings', Nickel's embedding can not be made without distortion for $\tau$-groups with rank of $\Gamma_c$ greater or equal to 2. 
\end{thm}
Before we give a proof, let's take a look at the same example as in Jennings'. 
So let $G$ be the free nilpotent group of rank $2$ and class $3$, with a Mal'cev basis $(y_1,y_2,y_3,y_4,y_5)$. 

We like to show that this group sits as a distorted subgroup of a unitriangular group under Nickel's embedding in all possible orderings on the basis elements for the $G$-submodule generated by $\{t_1, t_2, t_3, t_4, t_5 \} $. 

We first need to look at the actions of the each $y_i^k$ over each $t_j$ for $ 1 \leq i, j \leq 5$

\begin{align*} y_1^{a_1}y_2^{a_2}y_3^{a_3}y_4^{a_4}y_5^{a_5} y_1^{-k} & = y_1^{a_1-k}y_2^{a_2}y_3^{a_3+ka_2}y_4^{a_4+ka_3}y_5^{a_5} \\
 y_1^{a_1}y_2^{a_2}y_3^{a_3}y_4^{a_4}y_5^{a_5} y_2^{-k} & = y_1^{a_1}y_2^{a_2-k}y_3^{a_3}y_4^{a_4}y_5^{a_5+ka_3} \\
 y_1^{a_1}y_2^{a_2}y_3^{a_3}y_4^{a_4}y_5^{a_5} y_3^{-k} & = y_1^{a_1}y_2^{a_2}y_3^{a_3-k}y_4^{a_4}y_5^{a_5} \\
 y_1^{a_1}y_2^{a_2}y_3^{a_3}y_4^{a_4}y_5^{a_5} y_4^{-k} & = y_1^{a_1}y_2^{a_2}y_3^{a_3}y_4^{a_4-k}y_5^{a_5} \\
 y_1^{a_1}y_2^{a_2}y_3^{a_3}y_4^{a_4}y_5^{a_5} y_5^{-k} & = y_1^{a_1}y_2^{a_2}y_3^{a_3}y_4^{a_4}y_5^{a_5-k} \end{align*}
 
 Since all exponents are linear combination of $a_i$'s, we must have $ \{t_1, t_2, t_3,t_4,t_5, 1 \} $ as the $\Q$-basis for the $G$-submodule. Note that in order to obtain unitriangular presentation, we have to make sure $t_3$ should be on the left side of $t_4$ and $t_4$ should be on the right side of $t_5$ and identity $1$ should be placed to the far left end in the order of the basis. So we see that such an ordering is possible and there are more than one possibility. However in each ordering the weights of images of $y_4$ and $y_5$ can not be preserved at the same time. Because the only way to keep the images of the weights of these elements unchanged, both $t_4$ and $t_5$ should be placed at the same location which is only 2 elements away from the identity and this is not possible. Hence, under Nickel's embedding image of $G$ is always distorted.

\textit{Proof of Theorem~\ref{distorted1:thm}.} Let $G$ be a $\tau$-group with  $rank(\Gamma_c)=m$, where $m\geq 2$ and a Hirsch length $n$. Assume $\Gamma_c= \langle x_{n-m+1},x_{n-m}, \dots, x_n \rangle$. Then it is easy to see that , for $1\leq i \leq n$ and $ n-m+1\leq j \leq n$
$$t_i^{x_j ^k} = \begin{cases} x_j-k &\mbox{for}\:i=j \\  x_i &\mbox{for}\: i \neq j  \end{cases}$$
which means that the only way to accomplish to keep the weights of the images of the elements $\{ x_{n-m+1},x_{n-m}, \dots, x_n\} $ preserved is placing all $t_j$'s for $n-m+1\leq j \leq n$ at the same location; that is, $c-1$ elements away from the identity, which is clearly not possible. Therefore, under Nickel's embedding image of $G$ is always distorted. 

\section{Computing the distortion of subgroups}\label{subgroup}
Let $G$ be a finitely generated nilpotent group of class $c$ generated by a finite set $X$ and $H$ a subgroup of $G$. As before $\tau_i = \tau_i(G)= Is(\Gamma_i(G)),$ so $\tau_i(G) = T(G)$ for $i \geq c+1$.

Recall that for  $g \in G\smallsetminus T(G)$, the weight $\nu_{G}(g)$ is defined as the maximal $k$ such that $g \in \tau_k \smallsetminus \tau_{k+1}$.  Since 
$
G = \tau_1(G) \geq \tau_2(G)  \geq \ldots \geq \tau_{c+1}(G) = T(G)
$
the weight $\nu_G(g)$ is uniquely defined for every $g \in G\smallsetminus T(G)$. 

 For an element $g$ from a subgroup $H \leq G$ one can define the relative weight with respect to $H$ as
 $$
 r\nu_{G,H}(g) = \frac{\nu_G(g)}{\nu_H(g)}.
 $$
If the group $G$ and the subgroup $H$ are understood from the context we write $r\nu(g)$ instead of  $r\nu_{G,H}$.

As we have seen in Theorem~\ref{utnZR}, for a subgroup $H \leq G$ one has
 $$ \Delta_H^G (n) \sim n^{d}$$ where $$ d=\max_{h\in H\setminus\ T(H)} r\nu(h)$$
Clearly, this degree $d$ is uniquely defined by $H$ and we denote it by $d_H$.

Denote by $\CN_{r,c}$ the class of all nilpotent of class $c$ finite presentations $P = \langle X;R\rangle_c$ via generators $X$ and relators $R$ with $|X| \leq r$.

In this section we prove the following result.

\begin{thm}
Let $G$ be a finitely generated nilpotent group. Then there is a polynomial time algorithm that for a subgroup $H \leq G$, given by a finite generating set, computes the distortion degree $d_H$ of $H$ and finds an element $h \in H$ such that $r\nu(h) = d_H$. Furthermore, this algorithm is polynomial time uniformly in the nilpotency class $c$ and the size of the generating set of $G$, when the finite presentation of $G$ is a part of the input.
\end{thm}

\begin{proof} 
 We first describe the algorithm and then show that it has polynomial running time.
 
 Let $G$ be a finitely generated nilpotent  of class $c$  group given by a finite nilpotent of class $c$  presentation $P = \langle X;R\rangle_c$ with generators $X$  and relators $R$. In what follows we always assume that elelements of $G$ are given by words in $X \cup X^{-1}$ and  subgroups of $G$ are  given by some  finite generating sets. To $compute$ a subgroup always means to find some finite generating set of the subgroup.  
 
 In our constructions below we use polynomial time algorithms from \cite{MMO} that solve the following algorithmic problems:
 
\begin{itemize}
\item Subgroup Membership Problem: given a subgroup and an element of $G$ check if the element belongs to the subgroup.

\item Presentation Problem: given a subgroup of $G$ find a finite presentation for the subgroup, 
\item Intersection problem: given two subgroups of $G$ find their intersection;
\item Isolator problem: given a subgroup of $G$ find its isolator.

\end{itemize}
 
 Step 1. For each $ i = 1, \ldots,  c+1$ one can compute some finite generating sets of the subgroups $\tau_i(G)$. Indeed, it is known that the subgroup $\Gamma_i(G)$ is generated by all basic commutators of weight at least $i$ on $X$. The number of such commutators is bounded from above by a polynomial in $r = |X|$ and $c$. Since the Isolator problem in $G$ is decidable in polynomial time  we can compute the subgroup $\tau_i(G)$ in polynomial time uniformly in $r$ and $c$ (when the presentation of the group $G$ is a part of the input). 
 
 Suppose now that a subgroup $H$ of $G$ is given.
 
 Step 2.  Since the Presentation problem is decidable in Ptime, for the subgroup $H$ one can find in Ptime one of its finite presentations, and then, as was explained in Step 1, the subgroups 
 $H =  \tau_1(H) \geq \tau_2(H) \ldots \geq \tau_{c+1}(H) = T(H)$.

 Step 3.  Since the Intersection problem in $G$ is in Ptime  one can find in polynomial time a finite generating set, say $X_i$,  for  each subgroup $H_i = H\cap \tau_i(G)$, $i = c+1, c, \ldots, 1$. 
 
 Step 4. Using the Subgroup Membership problem one can find the largest  $m$ such that $H_m \not \leq T(G)$ (here $T(G) = \tau_{c+1}(G)$ is the torsion subgroup of  $G$). Denote this $m$ by $m_1$. It follows that for any $h \in H_{m_1} \smallsetminus T(G)$ one has $\nu_G(h) = m$.
 
 Step 5. Using the Subgroup Membership problem one can find the largest $t$ such that $H_{m_1} \leq \tau_t(H)$. Denote such $t$ by $t_1$. Observe, since  for any $i$ $H_{m_1}  \leq \tau_i(H)$ if and only if  for the  generating set $X_{m_1}$ of $H_{m_1}$ one has $X_{m_1} \leq \tau_t(H)$, one can find a particular element say $y_1 \in X_{m_1}$ such that $y_1 \in \tau_{t_1}(H) \smallsetminus \tau_{t_1+1}(H)$, so $\nu_H(y_1) = t_1$.  It follows that 
 $$
 r\nu(y_1) = \max \{r\nu(h) \mid h \in H_{m_1} \smallsetminus T(H)\}  = \frac{m_1}{t_1}.
$$

Step 6 (Loop). Similarly to  Step 4  one can find   the largest $m$ such that $H_m \not \leq H_{m_1}$. Denote this $m$ by $m_2$. It follows that for any $h \in H_{m_2} \smallsetminus H_{m_1}$ one has $\nu_G(h) = m_2$. 
Now similarly to Step 5 one can find the smallest $t$ such that $H_{m_2} \leq \tau_t(H)$. Denote such $t$ by $t_2$. Also, as in Step 5, one can find a particular element say $y_2 \in X_{m_2}$ such that  $\nu_H(y_2) = t_2$.  It follows that 
 $$
 r\nu(y_2) = \max \{r\nu(h) \mid h \in H_{m_2} \smallsetminus H_{m_1}\}  = \frac{m_2}{t_2}.
$$

Repeating this argument we construct a number $m_k$ and a series of subgroups 
$$
H = H_{m_k} \geq \ldots \geq H_{m_1} \geq T(H),
$$
and a series of elements $y_i \in H_{m_i}$ 
such that for any $i = 1, \ldots, k-1$ one has 

$$
 r\nu(y_{i+1}) = \max \{r\nu(h) \mid h \in H_{m_{i+1}} \smallsetminus H_{m_i}\}  = \frac{m_{i+1}}{t_{i+1}}
$$

Taking $j$ such that 
$$
\frac{m_j}{t_j} = \max \{\frac{m_i}{t_i} \mid i = 1, \ldots, k-1\}
$$
gives one the distortion degree $d_H$ of $H$ and also an element $y_j$ such that $r\nu(y_j) = d_h$. This proves the first statement of the theorem. The "furthermore" part follows from the construction since all the algorithms we use here are in Ptime uniformly in $|X|$ and $c$ when we consider the presentation $P = \langle X;R\rangle_c$ as part of the input to the algorithm (see \cite{MMNV,MMO}).

\end{proof}

\section{Open Questions}\label{OpenProblems}

We have seen that both Jennings' and Nickel's Embedding behave the same. Remember that when a proper Mal'cev basis is selected the image of a unitriangular group can be seen as an undistorted subgroup of a bigger unitriangular group. Moreover, when Heisenberg groups and $\tau$-groups with rank greater or equal to 2 are considered, the images under both Jennings's and Nickel's remain distorted. All these results lead to following question. 

\begin{itemize}
\item Given a $\tau$-group $G$, can we find an embedding that gives an undistorted image of $G$ in a unitraingular group. 

\end{itemize}

\end{document}